\documentclass[12pt,leqno]{amsart}
\usepackage{amsmath}
\usepackage{amssymb}
\def\overset#1#2{{\mathrel{\mathop {{#2}_{}}\limits^{#1}}}}
\def\underset#1#2{{\mathrel{\mathop {{}_{} {#2}}\limits_{{#1}_{}}}}}
\def\upplim_#1{\underset{#1}{\overline\lim}\;}
\def\lowlim_#1{\underset{#1}{\underline\lim}\;}
\newtheorem{corollary}[equation]{Corollary}

\newtheorem{lemma}[equation]{Lemma}
\newtheorem{proposition}[equation]{Proposition}

\newtheorem{theorem}[equation]{Theorem}

\newcommand{\C}{{\mathbf{C}}}

\newcommand{\N}{\mathbf{N}}

\renewcommand{\P}{{\mathbf{P}}}

\newcommand{\R}{{\mathbf{R}}}

\newcommand{\rank}{\mathrm{rank}}

\newcommand{\Z}{\mathbf{Z}}

\numberwithin{equation}{section}

\title[Second main theorem for meromorphic mappings]{Second main theorem for meromorphic mappings with moving hypersurfaces in subgeneral position} 

\author{Si Duc Quang}

\begin{document}

\maketitle 

\begin{abstract}
Let $f$ be an algebraically nondegenerate meromorphic mapping from $\C^m$ into $\P^n(\C)$ and let $Q_1,...,Q_q$ be $q$ hypersurfaces in $\P^n(\C)$ of degree $d_i$, in $N-$subgeneral position. In this paper, we will prove that, for every $\epsilon >0$, there exists a positive integer $M$ such that
$$||\ (q-(N-n+1)(n+1)-\epsilon) T_f(r)\le\sum_{i=1}^q\frac{1}{d_i}N^{[M]}(r,f^*Q_i)+o(T_f(r)).$$
Moreover, an explicit estimate for $M$ is given. Our result is an extension of the previous second main theorems for meromorphic mappings and moving hyperplanes or moving hypersurfaces.
\end{abstract}

\def\thefootnote{\empty}

\section{Introduction}
\vskip0.2cm 
Let $f$ be a meromorphic mapping from $\C^m$ into $\P^n(\C)$ with a reduced representation $\tilde f=(f_0,...,f_n)$. For each meromorphic mapping $a$ from $\C^m$ into $\P^n(\C)^*$, which is usually called a moving hyperplane, with a reduced representation $\tilde a=(a_0,...,a_n)$ such that $(\tilde f,\tilde a)=\sum_{j=0}^na_jf_j\not\equiv 0$, we denote by $f^*a$ the zero divisor of $(\tilde f,\tilde a)$. We see that $f^*a$ is defined independently from the choices of $\tilde f$ and $\tilde a$, and is called the intersecting divisor of $f$ with $a$. We denote by $N^{[M]}(r,f^*a)$ or $N^{[M]}_{(f,a)}(r)$ the counting function of $f^*a$ (see Section 2 for the definitions). As usual, we denote by $T_f(r)$ the characteristic function of $f$ with respect to the hyperplane line bundle of $\P^n(\C)$.
  
Let $\{a_i\}_{i=1}^q$ be moving hyperplanes of $\P^n(\C)$ with reduce representations $\tilde a_i=(a_{i0},\ldots ,a_{in})$. Let $N\geq n$ and $q\geq N+1$. We say that the family $\{a_i\}_{i=1}^q$ is in $N$-\textit{subgeneral position} if for every subset $R\subset \{1,2, \cdots, q\}$ with the cardinality $|R|=N+1$, 
$$ \rank_{\mathcal M}\{\tilde a_i\ |\ i\in R\}=n+1,$$
where $\mathcal M$ denotes the field consisting of all meromorphic functions on $\C^m$.
If they are in $n$-subgeneral position, we simply say that they are in {\it general position}. We also denote by $\mathcal K_{\{a_i\}_{i=1}^q}$ the smallest subfield of $\mathcal M$, which contains $\C$ and all $\frac{a_{ij}}{a_{ik}}$ for $a_{ik}\not\equiv 0$.

In 1991, W. Stoll - M. Ru \cite{RS1,RS2} and M. Shirosaki \cite{Sh90} independdently proved the following second main theorem.

\vskip0.2cm
\noindent
\textbf{Theorem A} (cf. \cite{RS1,RS2,Sh90}). {\it Let $\{a_i\}_{i=1}^q$ be $q$ hyperplanes of $\P^n(\C)$ in general position. Let $f:  {\C}^m \to \P^n(\C)$ be a meromorphic mapping such that $f$ is linearly nondegenerate over $\mathcal K_{\{a_i\}_{i=1}^q}$. Then for every $\epsilon >0$, we have}
$$||\ \ (q-n-1-\epsilon)T_f(r) \leq \sum_{i=1}^q N(r,f^*a_i)+ o(T_f(r)).$$
Here, by the notation ``$\| P$'' we mean that the assertion $P$ holds for all $r\in [0,\infty)$ excluding a Borel subset $E$ of the interval $[0,\infty)$ with $\int_E dr<\infty$.

The above second main theorem plays an important role in Nevanlinna theory, with many applications to Algebraic or Analytic geometry. We note that in the above result, the mapping $f$ is assumed to be linearly nondegenerate over the field $\mathcal K_{\{a_i\}_{i=1}^q}$. To treat the case where $f$ may be degenerate, we need consider the case where the hyperplanes may be not in general position, but in subgeneral position. Thanks the notion of Nochka weights introduced by Nochka \cite{Noc83}, D. D. Thai and S. D. Quang \cite{TQ} gave the following second main theorem for the case where the family of hyperplanes is in subgeneral position.

\vskip0.2cm
\noindent
\textbf{Theorem B} (cf. \cite{TQ}). {\it Let $f:\C^m \longrightarrow \P^n(\C)$ be a nonconstant meromorphic mapping. Let $\{a_i\}_{i=1}^{q}$ be meromorphic mappings of $\C^m$ into $\P^n(\C)^*$ in $N$-subgeneral position such that $a_i$ are slow with respect to $f$ and $f$ is linearly nondegenerate over $\mathcal{K}_{\{a_i\}_{i=1}^{q}}$.
Then for an arbitrary $0< \epsilon <1$,
$$||\ \ (q-2N+n-1-\epsilon)T_f(r)\leq \sum_{i=1}^{q}N^{[M]}(r,f^*a_i)+o(T_f(r)),$$
where $M$ is a positive integer (explicitly estimated).}

A natural question here is \textit{``how to generalize these results to the case where hyperplanes are replaced by hypersurfaces''}. By proposing a new technique (using a result of Corvaja and Zannier \cite{CZ} on the dimension of spaces of homogeneous polynomials), in 2004, M. Ru \cite{Ru04} proved a second main theorem for algebraically nondegenerate meromorphic mappings into $\mathbb P^{n}(\mathbb C)$ intersecting hypersurfaces in general position in $\mathbb P^{n}(\mathbb C)$. He proved the following.

\vskip0.2cm 
\noindent
\textbf{Theorem C} (cf. \cite{Ru04}). {\it Let $f:\C\to \P^n(\C)$ be an algebraically nondegenerate meromorphic mapping and let $Q_1,...,Q_q$ be $q$ hypersurfaces in $\P^n(\C)$ of degree $d_i$, in general position. Then, for every $\epsilon >0$,} 
$$||\ (q-n-1-\epsilon)T_f(r)\le\sum_{i=1}^q\frac{1}{d_i}N(r,f^*Q_i)+o(T_f(r)).$$

With the same assumptions, T. T. H. An and H. T. Phuong \cite{AP} improved the result of M. Ru by giving an explicit truncation level for counting functions. Recently, in \cite{Q16} we have generalized the results of M. Ru and T. T. H. An - H. T. Phuong to the following.

\vskip0.2cm 
\noindent
\textbf{Theorem D} (cf. \cite{Q16}). {\it Let $f:\C^m\to \P^n(\C)$ be an algebraically nondegenerate meromorphic mapping and let $Q_1,...,Q_q$ be hypersurfaces in $\P^n(\C)$ of degree $d_i$, in $N$-subgeneral position. Then, for every $\epsilon >0$, 
$$||\ (q-(N-n+1)(n+1)-\epsilon)T_f(r)\le\sum_{i=1}^q\frac{1}{d_i}N^{[M_0-1]}_{Q_i(f)}(r)+o(T_f(r)),$$
where $M_0$ is positive integer (explicitly estimated).}

For the case of moving hypersurface, recently G. Dethloff and T. V. Tan \cite{DT} generalized the second main theorems of M. Ru to the following.

\vskip0.2cm
\noindent
\textbf{Theorem E} (Dethloff - Tan \cite{DT})
{\it Let $f$ be a nonconstant meromorphic map of $\mathbf{C}^m$ into $\P^n(\mathbf{C})$. Let $\{Q_i\}_{i=1}^q$ be a set of slowly (with respect to $f$) moving hypersurfaces in weakly general position with $\deg Q_j = d_j\ (1\le i\le q).$ Assume that $f$ is algebraically nondegenerate over $\mathcal K_{\{Q_i\}_{i=1}^q}$.  Then for any $\epsilon >0$ there exist positive integers $L_j\ (j = 1,....,q)$, depending only on $n,\epsilon$ and $d_j\ (j = 1,...,q) $ in an explicit way such that 
$$  ||\ (q-n-1-\epsilon)T_f(r)\le \sum_{i=1}^{q}\dfrac{1}{d_i}N^{[L_j]}_{Q_i(f)}(r)+o(T_f(r)).$$}
Here, $\mathcal K_{\{Q_i\}}$ denotes the field generated by $\{Q_i\}_{i=1}^q$ (see Section 2 for the definition).

Our purpose in this paper is to generalize all these above mentioned results to the case of moving hypersurfaces in subgeneral position. We will prove a second main theorem for meromorphic mappings into $\P^n(\C)$ intersecting a family of moving hypersurfaces in subgeneral position with truncated counting functions. Namely, we will prove the following.

\begin{theorem}\label{thm1.1} 
Let $f$ be a nonconstant meromorphic map of $\mathbf{C}^m$ into $\P^n(\mathbf{C})$. Let $\{Q_i\}_{i=1}^q$ be a family of slowly (with respect to $f$) moving hypersurfaces in weakly $N-$subgeneral position with $\deg Q_j = d_j\ (1\le i\le q).$ Assume that $f$ is algebraically nondegenerate over $\mathcal K_{\{Q_i\}_{i=1}^q}$.  Then for any $\epsilon >0$, we have
$$  ||\ (q-(N-n+1)(n+1)-\epsilon)T_f(r)\le \sum_{i=1}^{q}\dfrac{1}{d_i}N^{[L_j]}_{Q_i(f)}(r)+o(T_f(r)),$$
where $L_j=\frac{1}{d_j}L_0$ and $L_0$ is a positive number which is defined by:
\begin{align*}
L_0&:=\binom{L+n}{n}p_0^{\binom{L+n}{n}\left (\binom{L+n}{n}-1\right )\binom{q}{n}-2}-1\\
\text{with }\ L&:=(n+1)d+2(N-n+1)(n+1)^3I(\epsilon^{-1})d\\
\text{ and }\ p_0&:=[\dfrac{\binom{L+n}{n}(\binom{L+n}{n}-1)\binom{q}{n}-1}{\log (1+\frac{\epsilon}{3(n+1)(N-n+1)})}]^2.
\end{align*}
\end{theorem}
Here, by $I(x)$ we denote the smallest integer which is not less than $x$. We see that, if the family of moving hypersurfaces is in general position, i.e., $N=n$, then our result will imply the second main theorem of G. Dethloff and T. V. Tan. Our idea to avoid using the Nochka weights here is that from every $N+1$ arbitrary moving hypersurfaces in weakly $N-$subgeneral position we will construct $n+1$ new moving hypersurfaces in weakly general position (see Lemma \ref{lem3.1}).

Let $Q$ be a moving hypersurface of $\P^n(\C)$. We define the truncated defect of $f$ with respect to $Q$ by
$$ \delta^{[L]}_f(D)=1-\lim\mathrm{inf}\frac{N^{[M]}(r,f^*Q)}{dT_f(r)}.$$
From the above theorem, we have the following defect relation for meromorphic mappings with a family of moving hypersurfaces as follows.
\begin{corollary}
Let $f$ be a nonconstant meromorphic map of $\mathbf{C}^m$ into $\P^n(\mathbf{C})$. Let $\{Q_i\}_{i=1}^q$ be a family of slowly (with respect to $f$) moving hypersurfaces in weakly $N-$subgeneral position with $\deg Q_j = d_j\ (1\le i\le q).$ Assume that $f$ is algebraically nondegenerate over $\mathcal K_{\{Q_i\}_{i=1}^q}$.  Then we have
$$\sum_{i=1}^q\delta^{[L_0]}_f(D)\le (N-n+1)(n+1).$$
\end{corollary}

\section{Basic notions and auxiliary results from Nevanlinna theory}

\noindent
{\bf 2.1. The first main theorem in Nevanlinna theory}

We set $||z|| = \big(|z_1|^2 + \dots + |z_m|^2\big)^{1/2}$ for
$z = (z_1,\dots,z_m) \in \mathbf{C}^m$ and define
\begin{align*}
B(r) := \{ z \in \mathbf{C}^m : ||z|| < r\},\quad
S(r) := \{ z \in \mathbf{C}^m : ||z|| = r\}\ (0<r<\infty).
\end{align*}
Define 
$$v_{m-1}(z) := \big(dd^c ||z||^2\big)^{m-1}\quad \quad \text{and}$$
$$\sigma_m(z):= d^c \text{log}||z||^2 \land \big(dd^c \text{log}||z||^2\big)^{m-1}
 \text{on} \quad \mathbf{C}^m \setminus \{0\}.$$

Let $F$ be a nonzero meromorphic function on a domain $\Omega$ in $\mathbf{C}^m$. For a set $\alpha = (\alpha_1,...,\alpha_m) $ of nonnegative integers, we set $|\alpha|=\alpha_1+...+\alpha_m$ and 
$$\mathcal {D}^\alpha F=\dfrac {\partial ^{|\alpha|} F}{\partial ^{\alpha_1}z_1...\partial ^{\alpha_m}z_m}.$$
We denote by $\nu^0_F, \nu^{\infty}_F$ and $\nu_F$ the zero divisor, the pole divisor, and the divisor of the meromorphic function $F$ respectively.



For a divisor $\nu$ on $\mathbf{C}^m$ and for a positive integer $M$ or $M= \infty$, we set
$$\nu^{[M]}(z)=\min\ \{M,\nu(z)\},$$
\begin{align*}
n(t) =
\begin{cases}
\int\limits_{|\nu|\,\cap B(t)}
\nu(z) v_{m-1} & \text  { if } m \geq 2,\\
\sum\limits_{|z|\leq t} \nu (z) & \text { if }  m=1. 
\end{cases}
\end{align*}
The counting function of $\nu$ is defined by
$$ N(r,\nu)=\int\limits_1^r \dfrac {n(t)}{t^{2m-1}}dt \quad (1<r<\infty).$$
Similarly, we define $N(r,\nu^{[M]})$ and denote it by $N^{[M]}(r,\nu)$.

Let $\varphi : \mathbf{C}^m \longrightarrow \mathbf{C} $ be a meromorphic function.
Define
$$N_{\varphi}(r)=N(r,\nu^0_{\varphi}), \ N_{\varphi}^{[M]}(r)=N^{[M]}(r,\nu^0_{\varphi}).$$
For brevity we will omit the character $^{[M]}$ if $M=\infty$.

Let $f : \mathbf{C}^m \longrightarrow \P^n(\mathbf{C})$ be a meromorphic mapping.
For arbitrarily fixed homogeneous coordinates
$(w_0 : \dots : w_n)$ on $\P^n(\mathbf{C})$, we take a reduced representation
$\tilde f = (f_0,  \ldots , f_n)$, which means that each $f_i$ is a  
holomorphic function on $\mathbf{C}^m$ and 
$f(z) = \big(f_0(z) : \dots : f_n(z)\big)$ outside the analytic set
$\{ f_0 = \dots = f_n= 0\}$ of codimension $\geq 2$.
Set $\Vert \tilde f \Vert = \big(|f_0|^2 + \dots + |f_n|^2\big)^{1/2}$.
The characteristic function of $f$ is defined by 
\begin{align*}
T_f(r)= \int\limits_{S(r)} \log\Vert \tilde f \Vert \sigma_m -
\int\limits_{S(1)}\log\Vert \tilde f\Vert \sigma_m.
\end{align*}

Let $\varphi$ be a nonzero meromorphic function on $\mathbf{C}^m$, which are occasionally regarded as a meromorphic map into $\P^1(\mathbf{C})$. The proximity function of $\varphi$ is defined by
$$m(r,\varphi):=\int_{S(r)}\log \max\ (|\varphi|,1)\sigma_m.$$
The Nevanlinna's characteristic function of $\varphi$ is defined as follows
$$ T(r,\varphi):=N_{\frac{1}{\varphi}}(r)+m(r,\varphi). $$
Then 
$$T_\varphi (r)=T(r,\varphi)+O(1).$$
The function $\varphi$ is said to be small (with respect to $f$) if $||\ T_\varphi (r)=o(T_f(r))$.

We denote by $\mathcal M$ (resp. $\mathcal K_f$) the field of all meromorphic functions (resp. small meromorphic functions with respect to $f$) on $\mathbf{C}^m$.

\vskip0.2cm 
\noindent
{\bf 2.2. Family of moving hypersurfaces}

We recall some following due to \cite{Q12,Q14}.

Denote by $\mathcal H_{\mathbf{C}^m}$ the ring of all holomorphic functions on $\mathbf{C}^m.$ Let $Q$ be a homogeneous polynomial in $\mathcal H_{\mathbf{C}^m}[x_0,\dots,x_n]$  of
degree $d \geq 1.$ Denote by $Q(z)$ the homogeneous  polynomial over $\mathbf{C}$ obtained by substituting a specific point $z \in \mathbf{C}^m$ into the coefficients of $Q$. We also call  a moving  hypersurface in $\P^n (\mathbf{C} )$  each homogeneous polynomial $Q \in\mathcal H_{\mathbf{C}^m}[x_0,\dots,x_n]$  such that the common zero set of all coefficients of $Q$ has codimension at least two.

Let $Q$ be a moving hypersurface in $\P^n(\mathbf{C})$ of degree $d\ge 1$ given by
$$ Q(z)=\sum_{I\in\mathcal T_d}a_I\omega^I, $$
where $\mathcal T_d=\{(i_0,...,i_n)\in\N_0^{n+1}\ ;\ i_0+\cdots +i_n=d\}$, $a_I\in\mathcal H_{\mathbf{C}^m}$ and $\omega^I=\omega_0^{i_0}\cdots\omega_n^{i_n}$. We consider the meromorphic mapping $Q':\mathbf{C}^m\rightarrow\P^N(\mathbf{C})$, where $N=\binom{n+d}{n}$, given by
$$ Q'(z)=(a_{I_0}(z):\cdots :a_{I_N}(z))\ (\mathcal T_d=\{I_0,...,I_N\}). $$
Here $I_0<\cdots<I_N$ in the lexicographic ordering. By chainging the homogeneous coordinates of $\P^n(\C)$ if neccesary, we may assume that for each given moving hypersurface as above, $a_{I_0}\not\equiv 0$ (note that $I_0=(0,\ldots,0,d)$ and $a_{I_0}$ is the coefficient of $\omega_n^d$). We set 
$$ \tilde Q=\sum_{j=0}^N\frac{a_{I_j}}{a_{I_0}}\omega^{I_j}. $$

The moving hypersurfaces $Q$ is said to be ``slow'' (with respect to $f$) if $||\ T_{Q'}(r)=o(T_f(r))$. This is equivalent to $||T_{\frac{a_{I_j}}{a_{I_0}}}(r)=o(T_f(r))\ (\forall 1\le j\le N)$, i.e., $\frac{a_{I_j}}{a_{I_0}}\in\mathcal K_f$.

Let $\{Q_i\}_{i=1}^q$ be a family of  moving hypersurfaces in $\P^n(\mathbf{C})$, $\deg Q_i=d_i$. Assume that
$$ Q_i=\sum_{I\in\mathcal T_{d_i}}a_{iI}\omega^I. $$
We denote by $\mathcal K_{\{Q_i\}_{i=1}^q}$ the smallest subfield of $\mathcal M$ which contains $\mathbf{C}$ and all $\frac{a_{iI}}{a_{iJ}}$ with $a_{iJ}\not\equiv 0$.  We say that $\{Q_i\}_{i=1}^q$ are in weakly $N$-subgeneral position $(N\ge n)$ if there exists $z \in \mathbf{C}^m$ such that all $a_{iI}\ (1\le i\le q,\ I\in\mathcal I)$ are holomorphic at $z$ and for any $1 \leq i_0 < \dots < i_N \leq q$ the system of equations
\begin{align*}
\left\{ \begin{matrix}
Q_{i_j}(z)(w_0,\dots,w_n) = 0\cr
0 \leq j \leq N\end{matrix}\right.
\end{align*}
has only the trivial solution $w = (0,\dots,0)$ in $\mathbf{C}^{n+1}$. If $\{Q_i\}_{i=1}^q$ is in weakly $n-$subgeneral position then we say that it is in weakly general position.

\vskip0.2cm
\noindent
\textbf{2.3. Some theorems and lemmas}

Let $f$ be a nonconstant meromorphic map of $\mathbf{C}^m$ into $\P^n(\mathbf{C})$. Denote by
$\mathcal C_{f}$ the set of all non-negative functions $h : \mathbf{C}^m\setminus A\longrightarrow [0,+\infty]\subset\overline\R$, which are of the form
$$ h=\dfrac{|g_1|+\cdots +|g_l|}{|g_{l+1}|+\cdots +|g_{l+k}|}, $$
where $k,l\in\N,\ g_1,...., g_{l+k}\in\mathcal K_f\setminus\{0\}$ and $A\subset\mathbf{C}^m$, which may depend on
$g_1,....,g_{l+k}$, is an analytic subset of codimension at least two. Then, for $h\in\mathcal C_{f}$ we have
$$\int\limits_{S(r)}\log h\sigma_m= o(T_f (r)).$$

\begin{lemma}[{see \cite{DT}}] \label{lem2.8}
Let $\{Q_i\}_{i=0}^n$ be a set of homogeneous polynomials of degree $d$ in $\mathcal K_f [x_0,..., x_n]$. Then there exists a function $h_1\in\mathcal C_{f}$ such that, outside an analytic set of $\mathbf{C}^m$ of codimension at least two,
$$ \max_{i\in\{0,...,n\}}|Q_i(f_0,...,f_n)|\le h_1||f||^d .$$
If, moreover, this set of homogeneous polynomials is in weakly general position, then there
exists a nonzero function $h_2\in\mathcal C_{f}$ such that, outside an analytic set of $\mathbf{C}^m$ of
codimension at least two,
$$h_2||f||^d \le  \max_{i\in\{0,...,n\}}|Q_i(f_0,...,f_n)|.$$
 \end{lemma}

\begin{lemma}[{Lemma on logarithmic derivative, see \cite{NO}}]\label{lem2.9}
Let $f$ be a nonzero meromorphic function on $\mathbf{C}^m.$ Then 
$$\biggl|\biggl|\quad m\biggl(r,\dfrac{\mathcal{D}^\alpha (f)}{f}\biggl)=O(\log^+T(r,f))\ (\alpha\in \Z^m_+).$$
\end{lemma}

Repeating the argument in (Prop. 4.5 \cite{Fu}), we have the following.

\begin{proposition}[{see \cite[Prop. 4.5]{Fu}}]\label{2.3}
Let $\Phi_1,...,\Phi_k$ be meromorphic functions on $\mathbf{C}^m$ such that $\{\Phi_1,...,\Phi_k\}$ 
are  linearly independent over $\mathbf{C}.$ Then  there exists an admissible set  
$$\{\alpha_i=(\alpha_{i1},...,\alpha_{im})\}_{i=1}^k \subset \Z^m_+$$
with $|\alpha_i|=\sum_{j=1}^{m}|\alpha_{ij}|\le i-1 \ (1\le i \le k)$ such that the following are satisfied:

(i)\  $\{{\mathcal D}^{\alpha_i}\Phi_1,...,{\mathcal D}^{\alpha_i}\Phi_k\}_{i=1}^{k}$ is linearly independent over $\mathcal M,$\ i.e., \ $\det{({\mathcal D}^{\alpha_i}\Phi_j)}\not\equiv 0,$ 

(ii) $\det \bigl({\mathcal D}^{\alpha_i}(h\Phi_j)\bigl)=h^{k}\cdot \det \bigl({\mathcal D}^{\alpha_i}\Phi_j\bigl)$ for
any nonzero meromorphic function $h$ on $\mathbf{C}^m.$
\end{proposition}

The next general form of second main theorem for hyperplanes is due to M. Ru \cite{Ru97}
\begin{theorem}[{see \cite[Theorem 2.3]{Ru97}}]\label{2.4}
Let $f$ be a linealy nondegenerate meromorphic mapping of $\C^m$ in $\P^n(\C)$ with a reduced representation $\tilde f=(f_0,...,f_n)$ and let $H_1,..,H_q$ be $q$ arbitrary hyperplanes in $\P^n(\C)$. Then we have
$$ ||\ \int_{S(r)}\max_{K}\log\left (\prod_{j\in K}\frac{||\tilde f||.||H_j||}{|H_j(\tilde f)|}\sigma_{m}\right)\le (n+1)T_f(r)-N_{W^\alpha (f_i)}(r)+o(T_f(r)),$$
where $\alpha$ is an admissible set with respect to $\tilde f$ (as in Proposition \ref{2.3}) and the maximum is taken over all subsets $K\subset\{1,...,q\}$ such that $\{H_j\ ;\ j\in K\}$ is linearly indenpendent.
\end{theorem}

\noindent
We note that the original theorem of M. Ru states only for the case of holomorphic curves  from $\C$. However its proof also is valid for the case of meromorphic mappings from $\C^m$ with a slight modification. 

We have some following algebraic lemmas due to \cite{CZ,DT}

\begin{lemma}[{see \cite[Lemma 2.2]{CZ}}]\label{2.5}
Let $A$ be a commutative ring and let $\{\phi_1,\ldots ,\phi_p\}$ be a regular sequence in $A$, i.e., for $i=1,\ldots ,p, \phi_i$ is not a zero divisor of $A/(\phi_1,\ldots ,\phi_{i-1})$. Denote by $I$ the ideal in $A$ generated by $\phi_1,\ldots ,\phi_p$. Suppose that for some $q,q_1,\ldots ,q_h\in A$, we have an equation
$$ \phi_1^{i_1}\cdots \phi_p^{i_p}\cdot q=\sum_{r=1}^h\phi_1^{j_1(r)}\cdots \phi_p^{j_p(r)}\cdot q_r, $$
where $(j_1(r),\ldots ,j_p(r))>(i_1,\ldots ,i_p)$ for $r=1,\ldots ,h$. Then $q\in I$.
\end{lemma}
Here, as throughout this paper, we use the lexicographic order on $\mathbf N_0^p$. Namely, 
$$(j_1,\ldots,j_p)>(i_1,\ldots,i_p)$$
iff for some $s\in\{1,\ldots ,p\}$ we have $j_l=i_l$ for $l<s$ and $j_s>i_s$.

\begin{lemma}[{see \cite[Lemma 3.2]{DT}}]\label{2.6}
Let $\{Q_i\}_{i=1}^q\ (q\ge n+1)$ be a set of homogeneous polynomials of common degree $d\ge 1$ in $\mathcal K_f[x_0,\ldots ,x_n]$ in weakly general position. Then for any pairwise different $1\le j_0,\ldots ,j_n\le q$ the sequence $\{Q_{j_0},\ldots,Q_{j_n}\}$ of elements in $\mathcal K_{\{Q_i\}}[x_0,\ldots,x_n]$ is a regular sequence, as well as all its subsequences.
\end{lemma}

\section{Second main theorems for moving hypersurfaces}

We first prove the following lemma.

\begin{lemma}\label{lem3.1}
Let $Q_1,...,Q_{N+1}$ be homogeneous polynomials in $\mathcal K_f[x_0,...,x_n]$ of the same degree $d\ge 1$, in weakly $N$-subgeneral position. Then there exist $n$ homogeneous polynomials $P_{2},...,P_{n+1}$ in $\mathcal K_f[x_0,...,x_n]$ of the forms
$$P_t=\sum_{j=2}^{N-n+t}c_{tj}Q_j, \ c_{tj}\in\C,\ t=2,...,n+1,$$
such that the family $\{P_1,...,P_{n+1}\}$ is in weakly general position, where $P_1=Q_1$.
\end{lemma}
\begin{proof} We assume that $Q_i\ (1\le i\le N+1)$ has the following form
$$ Q_i=\sum_{I\in\mathcal T_d}a_{iI}\omega^I.$$
By the definition of the weakly subgeneral position, there exists a point $z_0\in\C^m$ such that $a_{iI}$ is holomorphic at $z_0$ for all $i$ and $I$, and the following system of equations 
$$ Q_i(z_0)(\omega_0,...,\omega_n)=0, 1\le i\le N+1, $$
has only trivial solution $(0,...,0)$. We may assume that $Q_i(z_0)\not\equiv 0$ for all $1\le i\le N+1$. 

For each homogeneous polynomials $Q\in\C[x_0,...,x_n]$, we will denote by $Q^*$ the fixed hypersurface in $\P^n(\C)$ defined by $Q$, i.e.,
$$ Q^*=\{(\omega_0:\cdots :\omega_n)\in\P^n(\C)\ |\ Q(\omega_0,...,\omega_n)=0\}.$$
Setting $P_1=Q_1$, we see that
$$ \dim \left(\bigcap_{i=1}^tQ^*_i(z_0)\right)\le N-t+1,\ t=N-n+2,...,N+1,$$
where $\dim\emptyset =-\infty$.

Step 1. We will construct $P_2$ as follows. For each irreducible component $\Gamma$ of dimension $n-1$ of $Q^*_1(z_0)$, we put 
$$V_{1\Gamma}=\{c=(c_2,...,c_{N-n+2})\in\C^{N-n+1}\ ;\ \Gamma\subset Q^*_c(z_0),\text{ where }Q_c=\sum_{j=2}^{N-n+2}c_jQ_j\}.$$
Then $V_{1\Gamma}$ is a subspace of $\C^{N-n+1}$. Since $\dim \left(\bigcap_{i=1}^{N-n+2}Q^*_i(z_0)\right)\le n-2$, there exists $i\in\{2,...,N-n+2\}$ such that $\Gamma\not\subset Q^*_i(z_0)$. This implies that $V_{1\Gamma}$ is a proper subspace of $\C^{N-n+1}$. Since the set of irreducible components of dimension $n-1$ of $Q^*_1(z_0)$ is at most countable, 
$$ \C^{N-n+1}\setminus\bigcup_{\Gamma}V_{1\Gamma}\ne\emptyset. $$
Hence, there exists $(c_{12},...,c_{1(N-n+2)})\in\C^{N-n+1}$ such that
$$ \Gamma\not\subset P^*_2(z_0)$$
for all irreducible components of dimension $n-1$ of $Q^*_1(z_0)$, where
$P_2=\sum_{j=2}^{N-n+2}c_{1j}Q_j.$
This clearly implies that $\dim \left(P_1^*(z_0)\cap P_2^*(z_0)\right)\le n-2.$

Step 2. For each irreducible component $\Gamma'$ of dimension $n-2$ of $\left(P_1^*(z_0)\cap P_2^*(z_0)\right)$, put 
$$V_{2\Gamma'}=\{c=(c_2,...,c_{N-n+3})\in\C^{N-n+2}\ ;\ \Gamma\subset Q^*_c(z_0),\text{ where }Q_c=\sum_{j=2}^{N-n+3}c_jQ_j\}.$$
Hence, $V_{2\Gamma'}$ is a subspace of $\C^{N-n+2}$. Since $\dim \left(\bigcap_{i=1}^{N-n+3}Q_i^*(z_0)\right)\le n-3$, there exists $i, (2\le i\le N-n+3)$ such that $\Gamma'\not\subset Q_i^*(z_0)$. This implies that $V_{2\Gamma'}$ is a proper subspace of $\C^{N-n+2}$. Since the set of irreducible components of dimension $n-2$ of $\left(P_1^*(z_0)\cap P_2^*(z_0)\right)$ is at most countable, 
$$ \C^{N-n+2}\setminus\bigcup_{\Gamma'}V_{2\Gamma'}\ne\emptyset. $$
Then, there exists $(c_{22},...,c_{2(N-n+3)})\in\C^{N-n+2}$ such that
$$ \Gamma'\not\subset P_3^*(z_0) $$
for all irreducible components of dimension $n-2$ of $P_1^*(z_0)\cap P_2^*(z_0)$, where
$P_3=\sum_{j=2}^{N-n+3}c_{2j}Q_j.$
It is clear that $\dim \left(P_1^*(z_0)\cap P_2^*(z_0)\cap P_3^*(z_0)\right)\le n-3.$

Repeating again the above step, after the $n$-th step we get the hypersurfaces $P_2,...,P_{n+1}$ stisfying that
$$ \dim\left(\bigcap_{j=1}^tP_j^*(z_0)\right)\le n-t. $$
In particular, $\left(\bigcap_{j=1}^{n+1}P_j^*(z_0)\right)=\emptyset.$ This yields that $P_1,...,P_{n+1}$ are in weakly general position. We complete the proof of the lemma.
\end{proof}

\begin{proof}[{\sc Proof of Theorem \ref{thm1.1}}]
Replacing $Q_i$ by $Q_i^{d/d_i}$ if neccesary with the note that 
$$\dfrac{1}{d}N^{[L_0]}(r,f^*Q_i^{d/d_i})\le\frac{1}{d_i}N^{[L_j]}(r,f^*Q_i),$$
we may assume that all hypersurfaces $Q_i\ (1\le i\le q)$ are of the same degree $d$. We may also assume that $q>(N-n+1)(n+1)$. 

Consider a reduced representation $\tilde f=(f_0,\ldots ,f_n): \C\rightarrow \C^{n+1}$ of $f$. We also note that 
$$N^{[L_0]}_{Q_i(\tilde f)}(r)=N^{[L_0]}_{\tilde Q(\tilde f)}(r)+o(T_f(r)).$$
Then without loss of generality we may assume that $Q_i\in\mathcal K_f[x_0,...,x_n]$.

We set 
$$ \mathcal I=\{(i_1,...,i_{N+1})\ ; 1\le i_j\le q, i_j\ne i_t\ \forall j\ne t\}. $$
For each $I=(i_1,...,i_{N+1})\in\mathcal I$, we denote by $P_{I1},...,P_{I(n+1)}$ the hypersurfaces obtained in Lemma \ref{lem3.1} with respect to the family of hypersurfaces $\{Q_{i_1},...,Q_{i_{N+1}}\}$. It is easy to see that there exists a positive function $h\in\mathcal C_{f}$ such that
$$ |P_{It}(\omega)|\le h\max_{1\le j\le N+1-n+t}|Q_{i_j}(\omega)|, $$
for all $I\in\mathcal I$ and $\omega=(\omega_0,...,\omega_n)\in\C^{n+1}$.

For a fixed point $z\in \C\setminus\bigcup_{i=1}^qQ_i(\tilde f)^{-1}(\{0,\infty\})$.  We may assume that
$$ |Q_{i_1}(\tilde f)(z)|\le |Q_{i_2}(\tilde f)(z)|\le\cdots\le |Q_{i_{q}}(\tilde f)(z)|. $$
Let $I=(i_1,...,i_{N+1})$. Since $P_{I1},\ldots,P_{I(n+1)}$ are in weakly general position,  there exist functions $g_0,g\in\mathcal C_{f}$, which may be chosen independent of $I$ and $z$, such that
$$ ||\tilde f (z)||^d\le g_0(z)\max_{1\le j\le n+1}|P_{Ij}(\tilde f)(z)|\le g(z)|Q_{i_{N+1}}(z)|.$$
Therefore, we have
\begin{align*}
\prod_{i=1}^q\dfrac{||\tilde f (z)||^d}{|Q_i(\tilde f)(z)|}&\le g^{q-N}(z)\prod_{j=1}^{N}\dfrac{||\tilde f (z)||^d}{|Q_{i_j}(\tilde f)(z)|}\\
&\le g^{q-N}(z)h^{n-1}(z)\dfrac{||\tilde f (z)||^{Nd}}{\bigl (\prod_{j=2}^{N-n+1}|Q_{i_j}(\tilde f)(z)|\bigl )\cdot\prod_{j=1}^{n}|P_{Ij}(\tilde f)(z)|}\\
&\le g^{q-N}(z)h^{n-1}(z)\dfrac{||\tilde f (z)||^{Nd}}{|P_{I1}(\tilde f)(z)|^{N-n+1}\cdot\prod_{j=2}^{n}|P_{Ij}(\tilde f)(z)|}\\
&\le g^{q-N}(z)h^{n-1}(z)\zeta^{(N-n)(n-1)}(z)\dfrac{||\tilde f (z)||^{Nd+(N-n)(n-1)d}}{\prod_{j=1}^{n}|P_{Ij}(\tilde f)(z)|^{N-n+1}},
\end{align*}
where $I=(i_1,...,i_{N+1})$ and $\zeta$ is a function in $\mathcal C_{f}$, which is chosen common for all $I\in\mathcal I$, such that 
$$|P_{Ij}(z)(\omega)|\le \zeta (z) ||\omega||^d, \ \forall \omega=(\omega_0,...,\omega_n)\in\C^{n+1}.$$
The above inequality implies that
\begin{align}\label{3.2}
\log \prod_{i=1}^q\dfrac{||\tilde f (z)||^d}{|Q_i(\tilde f)(z)|}\le \log(g^{q-N}h^{n-1}\zeta^{(N-n)(n-1)})(z)+(N-n+1)\log \dfrac{||\tilde f (z)||^{nd}}{\prod_{j=1}^{n}|P_{Ij}(\tilde f)(z)|}.
\end{align}
Now, for each nonegative integer $L$, we denote by $V_L$ the vector space (over $\mathcal K_{\{Q_i\}}$) consisting of all homogeneous polynomials of degree $L$ in $\mathcal K_{\{Q_i\}}[x_0,\ldots ,x_n]$ and the zero polynomial. Denote by $(P_{I1},\ldots ,P_{In})$ the ideal in $\mathcal K_{\{Q_i\}}[x_0,\ldots ,x_n]$ generated by $P_{I1},\ldots ,P_{In}$.
\begin{lemma}[{see \cite[Lemma 5]{AP}, \cite[Proposition 3.3]{DT}}]\label{3.3}
Let $\{P_i\}_{i=1}^q\ (q\ge n+1)$ be a set of homogeneous polynomials of common degree $d\ge 1$ in $\mathcal K_f[x_0,\ldots ,x_n]$ in weakly general position. Then for any nonnegative integer $N$ and for any $J:=\{j_1,\ldots ,j_n\}\subset\{1,\ldots ,q\},$ the dimension of the vector space $\frac{V_L}{(P_{j_1},\ldots ,P_{j_n})\cap V_L}$ is equal to the number of $n$-tuples $(s_1,\ldots ,s_n)\in\mathbf N^n_0$ such that $s_1+\cdots +s_n\le L$ and $0\le s_1,\ldots,s_n\le d-1 $. In particular, for all $L\ge n(d-1)$, we have
$$ \dim\frac{V_L}{(P_{j_1},\ldots ,P_{j_n})\cap V_L}=d^n. $$
\end{lemma}


For each positive integer $L$ divisible by $d$ and for each $(i)=(i_1,\ldots,i_n)\in\mathbf N^n_0$ with $||(i)||=\sum_{s=1}^ni_s\le\frac{L}{d}$, we set
$$W^I_{(i)}=\sum_{(j)=(j_1,\ldots ,j_n)\ge (i)}P_{I1}^{j_1}\cdots P_{In}^{j_n}\cdot V_{L-d||(j)||}. $$
It is clear that $W^I_{(0,\ldots,0)}=V_L$ and $W^I_{(i)}\supset W^I_{(j)}$ if $(i)<(j)$ in the lexicographic ordering. Hence, $W^I_{(i)}$ is a filtration of $V_L$.
  
Let $(i)=(i_1,\ldots ,i_n),(i')=(i_1',\ldots ,i_n')\in \mathbf N^n_0$. Suppose that $(i')$ follows $(i)$ in the lexicographic ordering. 
We consider the following vector space homomorphism
$$ \varphi: \gamma\in V_{L-d||(i)||}\mapsto [P_{I1}^{i_1}\cdots P_{In}^{i_n}\gamma]\in\dfrac{W^I_{(i)}}{W^I_{(i')}}, $$
where $[P_{I1}^{i_1}\cdots P_{In}^{i_n}\gamma]$ is the equivalent class in $\frac{W^I_{(i)}}{W^I_{(i')}}$ containing $P_{I1}^{i_1}\cdots P_{In}^{i_n}\gamma$.
We see that $\varphi$ is surjective. We will show that $\ker\varphi$ is equal to $(P_{I1},\ldots ,P_{In})\cap V_{L-d||(i)||}$.

In fact, for any $\gamma\in\ker\varphi$, we have
\begin{align*}
P_{I1}^{i_1}\cdots P_{In}^{i_n}\gamma& =\sum_{(j)=(j_1,\ldots,j_n)\ge (i')}P_{I1}^{j_1}\cdots P_{In}^{j_n}||(j)||\\
&= \sum_{(j)=(j_1,\ldots,j_n)> (i)}P_{I1}^{j_1}\cdots P_{In}^{j_n}||(j)||,
\end{align*}
where $||(j)||\in V_{L-d||(j)||}$. By Lemma \ref{2.5} and Lemma \ref{2.6}, we have $\gamma\in (P_{I1},\ldots ,P_{In})$. Then
$$ \ker\varphi\subset (P_{I1},\ldots ,P_{In})\cap V_{L-d||(i)||}. $$
Conversely, for any $\gamma\in (P_{I1},\ldots ,P_{In})\cap V_{L-d||(i)||},\ (\gamma\ne 0)$, we have
$$ \gamma =\sum_{s=1}^nP_{Is} h_s,\ \ h_s\in V_{L-d(||(i)||+1)}. $$
It implies that
$$ \varphi (\gamma)=\sum_{s=1}^n[P_{I1}^{i_1}\cdots P_{I{s-1}}^{i_s-1}P_{Is}^{i_s+1}P_{I{s+1}}^{i_s+1}\cdots P_{In}^{i_n}h_s]. $$
It is clear that $P_{I1}^{i_1}\cdots P_{I{s-1}}^{i_s-1}P_{Is}^{i_s+1}P_{I{s+1}}^{i_s+1}\cdots P_{In}^{i_n}h_s\in W^I_{(i')}$, and hence $\varphi (\gamma)=0$, i.e., $\gamma\in\ker\varphi$. Therefore, we have
$$\ker\varphi = (P_{I1},\ldots ,P_{In})\cap V_{L-d||(i)||}.$$
This yields that
\begin{align}\label{3.4}
\dim \dfrac{W^I_{(i)}}{W^I_{(i')}}=\dim \dfrac{V_{L-d||(i)||}}{(P_{I1},\ldots ,P_{In})\cap V_{L-d||(i)||}}.
\end{align}

Fix a number $L$ large enough (chosen later). Set $u=u_L:=\dim V_L=\binom{L+n}{n}$. We assume that 
$$ V_L=W^I_{(i)_1}\supset W^I_{(i)_2}\supset\cdots\supset W^I_{(i)_K}, $$
where $W^I_{(i)_{s+1}}$ follows $W^I_{(i)_s}$ in the ordering and $(i)_K=(\frac{L}{d},0,\ldots ,0)$. It is easy to see that $K$ is the number of $n$-tuples $(i_1,\ldots,i_n)$ with $i_j\ge 0$ and $i_1+\cdots +i_n\le\frac{L}{d}$. Then we have
$$ K =\binom{\frac{L}{d}+n}{n}.$$
For each $k\in\{1,\ldots ,K-1\}$ we set $m^I_k=\dim \frac{W^I_{(i)_k}}{W^I_{(i)_{k+1}}}$, and set $m^I_k=1$. Then by Lemma \ref{3.5}, $m^I_k$ does not depends on $\{P_{I1},\ldots ,P_{In}\}$ and $k$, but on $||(i)_k||$. Hence, we set $m_k=m^I_{||(i)_k||}$. We also note that
\begin{align}\label{3.5}
m_k=d^n
\end{align}
 for all $k$ with $L-d||(i)_k||\ge n(d-1)$ (it is equivalent to $||(i)_k||\le\dfrac{L}{d}-n$).

From the above filtration, we may choose a basis $\{\psi^I_1,\cdots,\psi^I_u\}$ of $V_L$ such that  
$$\{\psi_{u-(m_s+\cdots +m_K)+1},\ldots ,\psi^I_u\}$$
 is a basis of $W^I_{(i)_s}$. For each $k\in\{1,\ldots,K\}$ and $l\in\{u-(m_k+\cdots +m_k)+1,\ldots, u-(m_{k+1}+\cdots +m_k)\}$, we may write
$$ \psi^I_l=P_{I1}^{i_{1k}}\cdots P_{In}^{i_{nk}}h_l,\ \text{ where } (i_{1k},\ldots,i_{nk})=(i)_k, h_l\in W^I_{L-d||(i)_k||}. $$
Then we have
\begin{align*}
|\psi^I_l(\tilde f)(z)|&\le |P_{I1} (\tilde f)(z)|^{i_{1k}}\cdots |P_{In} (\tilde f)(z)|^{i_{nk}}|h(\tilde f)(z)|\\
& \le c_l|P_{I1} (\tilde f)(z)|^{i_{1k}}\cdots |P_{In} (\tilde f)(z)|^{i_{nk}}||\tilde f(z)||^{L-d||(i)_k||}\\
&=c_l\left (\dfrac{|P_{I1} (\tilde f)(z)|}{||\tilde f(z)||^d}\right)^{i_{1k}}\cdots\left (\dfrac{|P_{In} (\tilde f)(z)|}{||\tilde f(z)||^d}\right)^{i_{nk}}||\tilde f (z)||^L,
\end{align*} 
where $c_l\in\mathcal C_{f}$, which does not depen on $f$ and $z$. Taking the product the both sides of the above inequalities over all $l$ and then taking logarithms, we obtain
\begin{align}\label{3.6}
\begin{split}
\log\prod_{l=1}^u|\psi^I_l(\tilde f)(z)|&\le\sum_{k=1}^Km_k\left (i_{1k}\log\dfrac{|P_{I1} (\tilde f)(z)|}{||\tilde f(z)||^d}+\cdots+i_{nk}\log\dfrac{|P_{In} (\tilde f)(z)|}{||\tilde f(z)||^d}\right)\\ 
& \ \ \ +uL\log ||\tilde f (z)||+\log c_I(z),
\end{split}
\end{align}
where $c_I=\prod_{l=1}^qc_l\in\mathcal C_{f}$, which does not depend on $f$ and $z$.

We see that
$$ \sum_{k=1}^Km_ki_{sk}=\sum_{l=0}^{\frac{L}{d}}\sum_{k:||(i)_k||=l}m(l)i_{sk}=\sum_{l=0}^{\frac{L}{d}}m(l)\sum_{k:||(i)_k||=l}i_{sk}. $$
Note that, by the symetry $(i_1,\ldots,i_n)\rightarrow (i_{\sigma (1)},\ldots ,i_{\sigma (n)})$ with $\sigma\in S(n)$,  $\sum_{k:||(i)_k||=l}i_{sk}$ does not depend on $s$. We set 
$$ A:= \sum_{k=1}^Km_ki_{sk},\ \text{ which is independent of $s$ and $I$}.$$
Hence, (\ref{3.6}) gives
\begin{align*}
\log\prod_{l=1}^u|\psi^I_l(\tilde f)(z)|&\le A\left (\log\prod_{i=1}^n\dfrac{|P_{Ii}(\tilde f)(z)|}{||\tilde f(z)||^d}\right)+uL\log ||\tilde f (z)||+\log c_I(z),
\end{align*}
i.e.,
\begin{align*}
A\left (\log\prod_{i=1}^n\dfrac{||\tilde f(z)||^d}{|P_{Ii}(\tilde f)(z)|}\right)\le\log\prod_{l=1}^u\frac{||\tilde f (z)||^L}{|\psi^I_l(\tilde f)(z)|}+\log c_I(z),
\end{align*}

Set $c_0=g^{q-N}h^{n-1}\zeta^{(N-n)(n-1)}\prod_{I}(1+c_I^{(N-n+1)/A})\in\mathcal C_{f}$. Combining the above inequality with (\ref{3.2}), we obtain that
\begin{align}\label{3.7}
\log \prod_{i=1}^q\dfrac{||\tilde f (z)||^d}{|Q_i(\tilde f)(z)|}\le \frac{N-n+1}{A}\log\prod_{l=1}^u\dfrac{||\tilde f (z)||^L}{|\psi^I_l(\tilde f)(z)|}+\log c_0.
\end{align}

We now write
$$ \psi^I_l=\sum_{J\in\mathcal T_N}c^I_{lJ}x^J\in V_L,\ \ c^I_{lJ}\in\mathcal K_{\{Q_i\}}, $$
where $\mathcal T_L$ is the set of all $(n+1)$-tuples $J=(i_0,\ldots,i_n)$ with $\sum_{s=0}^nj_s=L$, $x^J=x_0^{j_0}\cdots x_n^{j_n}$  and $l\in\{1,\ldots ,u\}$. For each $l$, we fix an index $J^I_l\in J$ such that $c^I_{lJ^I_l}\not\equiv 0$. Define
$$ \mu^I_{lJ}=\dfrac{c^I_{lJ}}{c^I_{lJ^I_l}},\  J\in\mathcal T_L.$$
Set $\Phi =\{\mu^I_{lJ};I\subset\{1,\ldots ,q\},\sharp I=n, 1\le l\le M, J\in\mathcal T_L\}$. Note that $1\in\Phi$. Let $B=\sharp\Phi$. We see that $B\le u\binom{q}{n}(\binom{L+n}{n}-1)=\binom{L+n}{n}(\binom{L+n}{n}-1)\binom{q}{n}$. For each positive integer $l$, we denote by $\mathcal L(\Phi (l))$ the linear span over $\C$ of the set 
$$\Phi (l)=\{\gamma_1\cdots\gamma_l;\gamma_i\in\Phi\}.$$
It is easy to see that
$$ \dim\mathcal L(\Phi(l))\le\sharp\Phi (l)\le\binom{B+l-1}{B-1}.$$
We may choose a positive integer $p$ such that
$$ p\le p_0:=[\dfrac{B-1}{\log (1+\frac{\epsilon}{3(n+1)(N-n+1)})}]^2\text{ and }\dfrac{\dim\mathcal L(\Phi (p+1))}{\dim\mathcal L(\Phi (p))}\le 1+\dfrac{\epsilon}{3(n+1)(N-n+1)}. $$
Indeed, if $\dfrac{\dim\mathcal L(\Phi (p+1))}{\dim\mathcal L(\Phi (p))}> 1+\dfrac{\epsilon}{3(n+1)(N-n+1)}$ for all $p\le p_0$, we have 
$$\dim\mathcal L(\Phi (p_0+1))\ge (1+\dfrac{\epsilon}{3(n+1)(N-n+1)})^{p_0}.$$ 
Therefore, we have
\begin{align*}
\log (1+\dfrac{\epsilon}{3(n+1)(N-n+1)})&\le\dfrac{\log \dim\mathcal L(\Phi (p_0+1))}{p_0}\le\dfrac{\log \binom{B+p_0}{B-1}}{p_0}\\ 
& =\dfrac{1}{p_0}\log \prod_{i=1}^{B-1}\dfrac{p_0+i+1}{i}<\dfrac{(B-1)\log (p_0+2)}{p_0}\\
&\le \dfrac{B-1}{\sqrt{p_0}}\le \dfrac{(B-1)\log (1+\frac{\epsilon}{3(n+1)(N-n+1)})}{B-1}\\
&=\log (1+\frac{\epsilon}{3(n+1)(N-n+1)}).
\end{align*}
This is a contradiction.

We fix a positive integer $p$ satisfying the above condition. Put $s=\dim\mathcal L(\Phi (p))$ and $t=\dim\mathcal L(\Phi (p+1))$. Let $\{b_1,\ldots,b_t\}$ be an $\C$-basis of $\mathcal L(\Phi (p+1))$ such that $\{b_1,\ldots ,b_s\}$ be a $\C$-basis of $\mathcal L(\Phi (p))$.

For each $l\in\{1,\ldots ,u\}$, we set 
$$ \tilde\psi^I_l =\sum_{J\in\mathcal T_N}\mu^I_{lJ}x^I.$$
For each $J\in\mathcal T_L$, we consider homogeneous polynomials $\phi_J(x_0,\ldots ,x_n)=x^J$. Let $F$ be a meromorphic mapping of $\C^m$ into $\P^{tu-1}(\C)$ with a reduced representation $\tilde F = (b_i\phi_J(\tilde f))_{1\le i\le t,J\in\mathcal T_N}$ be a meromorphic vector functions from $\C^m$ into $\C^{tu}$. Since $f$ is assumed to be algebraically nondegenerate over $\mathcal K_{\{Q_i\}}$, $F$ is linearly nondegenerate over $\C$.  We see that there exist nonzero functions $c_1,c_2\in\mathcal C_{f}$ such that 
$$c_1||\tilde f||^L\le ||\tilde F||\le c_2||\tilde f||^L.$$

For each $l\in\{1,\ldots ,u\}, 1\le i\le s$, we consider the linear form $L^I_{il}$ in $x^J$ such that 
$$ b_i\tilde\psi^I_l(\tilde f)=L^I_{il}(\tilde F). $$
Since $f$ is algebraically nondegenerate over $\mathcal K_{\{Q_i\}}$, it is easy to see that $\{b_i\tilde\psi^I_l(\tilde f); 1\le i\le s,1\le l\le M\}$ is linearly independently over $\C$, and so is $\{L^I_{il}(\tilde F);1\le i\le s,1\le l\le M\}$. This yields that $\{L^I_{il};1\le i\le s,1\le l\le M\}$ is linearly independent over $\C$. We also see that
\begin{align*}
s\log\prod_{l=1}^u\dfrac{||\tilde f (z)||^L}{|\psi^I_l(\tilde f)(z)|}&=\log\prod_{\overset{1\le l\le u}{1\le i\le s}}\dfrac{||\tilde F (z)||}{|b_i\psi^I_l(\tilde f)(z)|}+\log c_3(z)\\
&=\log\prod_{\overset{1\le l\le u}{1\le i\le s}}\dfrac{||\tilde F (z)||\cdot ||L^I_{il}||}{|L^I_{il}(\tilde F)(z)|}+\log c_4(z),
\end{align*}
where $c_3,c_4$ are nonzero functions in $\mathcal C_{f}$, not depend on $f$ and $I$, but on $\{Q_i\}_{i=1}^q$.
Combining this inequality and (\ref{3.7}), we obtain that
\begin{align}\label{3.8}
\log \prod_{i=1}^q\dfrac{||\tilde f (z)||^d}{|Q_i(\tilde f)(z)|}\le \frac{N-n+1}{sA}\left (\max_{I}\log\prod_{\overset{1\le l\le u}{1\le i\le s}}\dfrac{||\tilde F (z)||\cdot ||L^I_{il}||}{|L^I_{il}(\tilde F)(z)|}+\log c_4(z)\right)+\log c_0(z).
\end{align}

Since $\tilde F$ is linearly nondegenerate over $\C$, there exists an admissible set $\alpha =(\alpha_{iJ})_{\overset{1\le i\le t}{J\in\mathcal T_N}}$ with $\alpha_{iJ}\in \Z_+^m$, $||\alpha_{iJ}||\le tu-1.$ such that
$$ W^{\alpha}(b_i\tilde\phi_J(\tilde f))=\det\left (\mathcal D^{\alpha_{i'J'}}(b_i\tilde\phi_J(\tilde f))\right)\not\equiv 0.$$
By Theorem \ref{2.4}, we have
\begin{align}\label{3.9}
\Big|\Big|\ \int\limits_{S(r)}\max_{I}\left\{\log\prod_{\overset{1\le l\le u}{1\le i\le s}}\dfrac{||\tilde F (z)||\cdot ||L^I_{il}||}{|L^I_{il}(\tilde F)(z)|}\right\}&\le tuT_F(r)-N_{W^{\alpha}(b_i\tilde\phi_J(\tilde f))}(r)+o(T_F(r)).
\end{align}
Integrating both sides of (\ref{3.8}) and using (\ref{3.9}), we obtain that 
\begin{align}\label{3.10}
\begin{split}
qdT_f(r)-\sum_{i=1}^qN(r,f^*Q_i)\le&\frac{tu(N-n+1)}{sA}T_F(r)-\frac{N-n+1}{sA}N_{W^{\alpha}(b_i\tilde\phi_J(\tilde f))}(r)\\
&+o(T_F(r)+T_f(r)).
\end{split}
\end{align}

We now estimate the quantity $\sum_{i=1}^qN(r,f^*Q_i)-\frac{N-n+1}{sA}N_{W^{\alpha}(b_i\tilde\phi_J(\tilde f))}(r)$. Fix a point $z_0\in\C^m$. Without lose of generality, we may assume that
$$ \nu_{Q_1(\tilde f)}(z_0)\ge\cdots\ge\nu_{Q_N(\tilde f)}(z_0)\ge\cdots\ge\nu_{Q_q(\tilde f)}(z_0). $$
First, we recall that
$$ Q_i(x)=\sum_{J\in\mathcal T_d}a_{iJ}x^I\in\mathcal K_{\{Q_i\}}[x_0,\ldots ,x_n]. $$
Let $T=(\cdots ,t_{kJ},\cdots )\ (k\in\{1,\ldots ,q\}, J\in\mathcal T_d)$ be a family of variables and
$$  Q^T_i=\sum_{J\in\mathcal T_d}T_{iJ}x^I\in\mathbf Z[T,x],\ \ i=1,\ldots ,q. $$
For each ordered subset $I=(i_1,...,i_{N+1})\subset\{1,\ldots ,q\}$, we denote by $\tilde R_I\in\mathbf Z[T]$ the resultant of $\{Q^T_i\}_{i\in I}$. Then there exists a positive integer $\lambda$ (common for all $I$) and polynomials $\tilde b_{ij}^I\ (0\le i\le n, j\in I)$ in $\mathbf Z[T,x]$, which are zero or homogeneous in $x$ with degree of $\lambda -d$ such that 
$$ x_i^\lambda\cdot\tilde R_I=\sum_{j\in I}\tilde b_{ij}^IQ^T_j\ \text{ for all }i\in\{1,\ldots ,q\},$$
and $R_I=\tilde R_I(\ldots, a_{kJ},\ldots )\not\equiv 0$. We see that $R_I\in\mathcal K_{f}$. Set
$$b^I_{ij}=\tilde b^I_{ij}((\ldots ,a_{jJ},\ldots ),(x_0,\ldots ,x_n)).$$
Then we have
$$ f^\lambda_i\cdot R_I=\sum_{j\in H}b^I_{ij}(\tilde f)Q_j(\tilde f)\ \text{ for all }i\in\{0,\ldots ,n\}.$$
This implies that
$$ \nu_{R_I}\ge\min_{j\in I}\nu_{Q_j(\tilde f)}+\min_{0\le i\le n, j\in I}\nu_{b^I_{ij}(\tilde f)}.$$
We set $R=\prod_{I\subset\{1,\ldots ,q\}}R_I\in\mathcal K_{\{Q_i\}}$. It is easy to see that
$$ \nu_{b^I_{ij}(\tilde f)}\ge O(\min_{k,J}\nu_{a_{kJ}}), $$
and the left hand side of this inequality is only depend on $\{Q_i\}$. Then it implies that there exists a constant $c$, which depends only on $\{Q_i\}$, such that
$$ \min_{j\in I}\nu_{Q_j(f)}\le \nu_{R}-c\min_{k,J}\nu_{a_{kJ}}, $$
for each ordered subset $I\subset\{1,\ldots ,q\}$ with $\sharp I=N+1$.

Now, we let $I=\{1,...,N+1\}\subset\{1,\ldots ,q\}$. Then
$$ \nu_{Q_{j}(f)}(z_0)\le \nu_{R}(z_0)-c\min_{k,J}\nu_{a_{kJ}}(z_0),\ j=N+1,...,q. $$
Also it is easy to see that
$$ \nu_{Q_{N-n+i}(\tilde f)}(z_0)\le\nu_{P_{Ii}}(z_0),$$
and hence
$$ \nu_{Q_{N-n+i}(\tilde f)}(z_0)-\nu^{[tu-1]}_{Q_{N-n+i}(\tilde f)}(z_0)\le \nu_{P_{Ii}}(z_0)-\nu^{[tu-1]}_{P_{Ii}}(z_0),\ i=2,...,n.$$
Therefore,
\begin{align}\nonumber
\sum_{i=1}^q(\nu_{Q_i(\tilde f)}(z_0)-\nu^{[tu-1]}_{Q_{i}(\tilde f)}(z_0))
&\le (N-n+1)(\nu_{Q_1(\tilde f)}(z_0)-\nu^{[tu-1]}_{Q_1(\tilde f)}(z_0))\\
\begin{split}\label{3.11}
&+\sum_{i=2}^n(\nu_{P_{Ii}(\tilde f)}(z_0)-\nu^{[tu-1]}_{P_{Ii}(\tilde f)}(z_0))+(q-N)\nu_{Q_{N+1}(\tilde f)}(z_0)\\
&\le (N-n+1)\sum_{i=1}^n(\nu_{P_{Ii}(\tilde f)}(z_0)-\nu^{[tu-1]}_{P_{Ii}(\tilde f)}(z_0))\\
&+(q-N)(\nu_{R}(z_0)-c\min_{k,J}\nu_{a_{kJ}}(z_0)).
\end{split}
\end{align} 

Take linear forms $h^I_{il}$ in $x^J$, $1\le l\le u, s+1\le i\le t, J\in\mathcal T_L$ such that $\{L^I_{il};1\le l\le u, 1\le i\le s\}\cup\{h^I_{il}; 1\le l\le u, s+1\le i\le t\}$ is linearly independent over $\C$. Moreover, we easily see that
\begin{align}\label{3.12}
\begin{split}
\nu_{W^{\alpha}(b_i\tilde\phi_J(\tilde f))}(z_0)&=\nu_{W^{\alpha}(L^I_{il}(\tilde f),...,h^I_{il}(\tilde f))}(z_0)\\
&\ge\sum_{\overset{1\le l\le u}{1\le i\le s}}\left (\nu_{L^I_{il}(\tilde f)}(z_0)-\nu^{[tu-1]}_{L^I_{il}(\tilde f)}(z_0)\right )\\
&\ge \sum_{\overset{1\le l\le u}{1\le i\le s}}\left (\nu_{b_i\tilde\psi^I_{il}(\tilde f)}(z_0)-\nu^{[tu-1]}_{b_i\tilde\psi^I_{il}(\tilde f)}(z_0)\right )\\
&\ge \sum_{\overset{1\le l\le u}{1\le i\le s}}\left (\nu_{\tilde\psi^I_{il}(\tilde f)}(z_0)-\nu^{[tu-1]}_{\tilde\psi^I_{il}(\tilde f)}(z_0)\right )-C\max_{1\le i\le s}\nu^{\infty}_{b_i}(z_0),
\end{split}
\end{align}
where $C$ is a positive constant, which is chosen indenpendently of $I$, since there are only finite ordered subset $I$.

Now for integers $x,y$, we easily see that 
\begin{align*}
\max\{0,x+y-L\}\ge \max\{0,x-L\}+\min\{0,y\}.
\end{align*} It yields that
\begin{align}\label{3.13}
\nu_{\varphi_1\varphi_2}(z)-\nu_{\varphi_1\varphi_2}^{[L]}(z)\ge\nu_{\varphi_1}(z)-\nu_{\varphi_1}^{[L]}(z)-\nu^{\infty}_{\varphi_2}(z),
\end{align}
for every nonzero meromorphic functions $\varphi_1,\varphi_2$. If let $x_1,\ldots ,x_{k_1}$ be $k_1$ nonegative integers and let $y_1,\ldots ,y_{k_2}$ be $k_2$ negative integers, then we have the following estimate
\begin{align*}
\max\{0,x_1&+\cdots +x_{k_1}+y_1+\cdots +y_{k_2}-L\} \ge\sum_{i=1}^{k_1}\max\{0,x_i-L\}+\sum_{i=1}^{k_2}\min\{0,y_i\}\\
&=\sum_{i=1}^{k_1}(\max\{0,x_i-L\}+\min\{0,x_i\})+\sum_{i=1}^{k_2}(\max\{0,y_i-L\}+\min\{0,y_i\}).
\end{align*}
This yields that
\begin{align}\label{3.14}
\nu_{\prod_{i=1}^k\varphi_i}(z)-\nu^{[L]}_{\prod_{i=1}^k\varphi_i}(z)\ge\sum_{i=1}^k(\nu_{\varphi_i}(z)-\nu^{[L]}_{\varphi_i}(z)-\nu^{\infty}_{\varphi_i}(z)),
\end{align}
for any meromorphic functions $\varphi_i \ (1\le i\le k)$.

For each $1\le l\le u, 1\le i\le s$ we have
\begin{align*}
\tilde\psi^I_{l}(\tilde f)=\frac{1}{c^I_{lJ_l^I}}\prod_{j=1}^nP_{Ij}^{i_{jk}}(\tilde f)h_l(\tilde f),
\end{align*}
where $(i_{1k},\ldots ,i_{nk})=I_k, h_l\in V_{L-d||(i)_k||}$ and $h_l$ is independent of $f$.
Now using (\ref{3.13}) and (\ref{3.14}), we have
\begin{align*}
\nu_{\tilde\psi^I_{il}(\tilde f)}(z_0)-\nu^{[tu-1]}_{\tilde\psi^I_{il}(\tilde f)}(z_0)&\ge\nu_{\prod_{j=1}^nP^{i_{jk}}_{Ij}(\tilde f)}(z)-\nu^{[tu-1]}_{\prod_{j=1}^nP^{i_{jk}}_{Ij}(\tilde f)}(z)-\nu_{c^I_{lJ_l^I}}(z_0)\\ 
& \ge\sum_{j=1}^ni_{jk}(\nu_{P_{Ij}(\tilde f)}(z)-\nu^{[tu-1]}_{P_{Ij}(\tilde f)}(z))-c_1\max_{j,J}\nu_{a_{jJ}}(z_0).
\end{align*}
where $c_1$ is a constant, which depends only on $\{Q_i\}, t$ and $L$. Summing-up both sides of the above inequalities over all $1\le i\le u,1\le l\le s$, we get
\begin{align}
\nonumber
\sum_{\overset{1\le i\le u}{1\le l\le s}}(\nu_{\tilde\psi^I_{il}(\tilde f)}(z_0)-\nu^{[tu-1]}_{\tilde\psi^I_{il}(\tilde f)}(z_0))\ge&\sum_{j=1}^ns\sum_{k=1}^Km^I_ki_{jk}(\nu_{P_{Ij}(\tilde f)}(z_0)-\nu^{[tu-1]}_{P_{Ij}(\tilde f)}(z_0))-c_{2}\max_{j,J}\nu_{a_{jJ}},\\
\label{3.15}
=&As\sum_{j=1}^n(\nu_{P_{Ij}(\tilde f)}(z)-\nu^{[tu-1]}_{P_{Ij}(\tilde f)}(z))-c_{2}\max_{j,J}\nu_{a_{jJ}},
\end{align}
where $c_{2}$ is a constant, which depends only on $\{Q_i\}, t$ and $L$.

Combining (\ref{3.12}) and (\ref{3.15}), we get
\begin{align*}
\nu_{W^{\alpha}(b_i\tilde\phi_J(\tilde f))}(z_0)\ge As\sum_{j=1}^n(\nu_{P_{Ij}(\tilde f)}(z)-\nu^{[tu-1]}_{P_{Ij}(\tilde f)}(z))-c_{2}\max_{j,J}\nu_{a_{jJ}} -C\max_{1\le i\le s}\nu^{\infty}_{b_i}(z_0).
\end{align*} 
Combining (\ref{3.11}) and this inequality, we obtain
\begin{align*}
\frac{N-n+1}{As}\nu_{W^{\alpha}(b_i\tilde\phi_J(\tilde f))}(z_0)\ge&\sum_{i=1}^q(\nu_{Q_i(\tilde f)}(z_0)-\nu^{[tu-1]}_{Q_{i}(\tilde f)}(z_0))-(N-n+1)(q-N)(\nu_{R}(z_0)\\
+&c\max_{1\le k\le q}\nu_{a_{kJ_k}}(z_0))-\frac{N-n+1}{As}(c_{2}\max_{j,J}\nu_{a_{jJ}} +C\max_{1\le i\le s}\nu^{\infty}_{b_i}(z_0)).
\end{align*}
Integrating both sides of the above inequality, we obtain that
$$||\ \frac{N-n+1}{As}N_{W^{\alpha}(b_i\tilde\phi_J(\tilde f))}(r)\ge \sum_{i=1}^q(N_{Q_i(\tilde f)}(r)-N^{[tu-1]}_{Q_{i}(\tilde f)}(r))+o(T_f(r)).$$
From this inequality and (\ref{3.10}) with a note that $T_F(r)=LT_f(r)+o(T_f(r))$, we have
\begin{align}\label{3.16}
||\ (q-\frac{tuL(N-n+1)}{dAs})T_f(r)\le\sum_{i=1}^q\frac{1}{d}N^{[tu-1]}(r,f^*Q_i)+o(T_f(r)).
\end{align}

Now we give some estimates for $A$, $t$ and $s$. For each $I_k=(i_{1k},\ldots ,i_{nk})$ with $||(i)_k||\le \frac{L}{d}-n$, we set 
$$i_{(n+1)k}=\dfrac{L}{d}-n-\sum_{s=1}^ni_s.$$
 Since the number of nonnegative integer $p$-tuples with summation $\le T$ is equal to the number of nonnegative
integer $(p+1)$-tuples with summation exactly equal to $T\in\mathbf Z$, which is $\binom{T+n}{n}$, and since the sum below is independent of $s$, we have
\begin{align*}
A&=\sum_{||(i)_k||\le\frac{L}{d}}m^I_ki_{sk}\ge \sum_{||(i)_k||\le\frac{L}{d}-n}m^I_ki_{sk}=\dfrac{d^n}{n+1}\sum_{||(i)_k||\le\frac{L}{d}-n}\sum_{s=1}^{n+1}i_{sk}\\
&=\dfrac{d^n}{n+1}\cdot \binom{\frac{L}{d}}{n}\cdot (\dfrac{L}{d}-n)=d^n\binom{\frac{L}{d}}{n+1}.
\end{align*}

Now, for every positive number $x\in[0,\frac{1}{(n+1)^2}]$, we have
\begin{align}\label{new1.2}
\begin{split}
(1+x)^{n}&=1+nx+\sum_{i=2}^n\binom{n}{i}x^{i}\le 1+nx+\sum_{i=1}^2\frac{n^{i}}{i!(n+1)^{2i-2}}x\\
&=1+nx+\sum_{i=2}^n\frac{1}{i!}x\le 1+(n+1)x.
\end{split}
\end{align}
We chose $L= (n+1)d+2(N-n+1)(n+1)^3I(\epsilon^{-1})d$. Then $L$ is divisible by $d$ and we have
\begin{align}\label{new1.3}
\frac{(n+1)d}{L-(n+1)d}=\frac{(n+1)d}{2(N-n+1)(n+1)^3I(\epsilon^{-1})d}\le \frac{1}{2(n+1)^2}.
\end{align}
Therefore, using (\ref{new1.2}) and (\ref{new1.3}) we have
\begin{align*}
\dfrac{uL}{dA}&\le \dfrac{\binom{L+n}{n}L}{d^{n+1}\binom{\frac{L}{d}}{n+1}}
=\frac{L\cdot (L+1)\cdots (L+n)}{1\cdot 2\cdots n}\Big /\frac{(L-nd)\cdot (L-(n-1)d)\cdots L}{1\cdot 2\cdots (n+1) }\\
&=(n+1)\prod_{i=1}^n\dfrac{L+i}{L-(n-i+1)d}<(n+1)\bigl (\dfrac{L}{L-(n+1)d}\bigl)^n\\
&=(n+1)\left (1+\dfrac{(n+1)d}{L-(n+1)d}\right )^n<(n+1)\left (1+\frac{(n+1)^2d}{2(N-n+1)(n+1)^3I(\epsilon^{-1})d}\right )\\
&\le (n+1)+\dfrac{(n+1)^3d}{2(n+1)^3(N-n+1)\epsilon^{-1}}\le n+1+\dfrac{\epsilon}{2(N-n+1)}.
\end{align*}
Then we have
\begin{align}\label{3.17}
\begin{split}
\dfrac{tuL}{dAs}&\le (1+\dfrac{\epsilon}{3(n+1)(N-n+1)})(n+1+\dfrac{\epsilon}{2(N-n+1)})\\
&\le n+1+\dfrac{\epsilon}{2(N-n+1)}+\dfrac{\epsilon}{3(N-n+1)}+\dfrac{\epsilon}{6(N-n+1)}\\
&=n+1+\frac{\epsilon}{N-n+1}.
\end{split}
\end{align}

Combining (\ref{3.11}) and (\ref{3.17}), we get
\begin{align}
\label{3.18}
\left (q-(N-n+1)(n+1)-\epsilon\right )T_f(r)\le \sum_{i=1}^q\frac{1}{d}N^{[tu-1]}(r,f^*Q_i)+o(T_f(r)).
\end{align}
Here we note that:
\begin{itemize}
\item $L:=(n+1)d+2(N-n+1)(n+1)^3I(\epsilon^{-1})d$,
\item $p_0:=\left [\dfrac{B-1}{\log (1+\frac{\epsilon}{3(n+1)(N-n+1)})}\right ]^2\le \left [\dfrac{\binom{L+n}{n}(\binom{L+n}{n}-1)\binom{q}{n}-1}{\log (1+\frac{\epsilon}{3(n+1)(N-n+1)})}\right ]^2,$
\item $tu-1\le\binom{L+n}{n}\binom{B+p}{B-1}-1\le \binom{L+n}{n}p^{B-1}-1\le \binom{L+n}{n}p_0^{\binom{L+n}{n}(\binom{L+n}{n}-1)\binom{q}{n}-2}-1=L_0.$
\end{itemize}
By these estimates and by (\ref{3.18}), we obtain
\begin{align*}
||\ (q-(N-n+1)(n+1)-\epsilon)T_f(r)\le \sum_{i=1}^q\frac{1}{d}N^{[L_0]}(r,f^*Q_i)+o(T_f(r)).
\end{align*}
The theorem is proved.
\end{proof}

\noindent
{\bf Acknowledgements.} This research is funded by Vietnam National Foundation for Science and Technology Development (NAFOSTED) under grant number 101.04-2015.03.

\vskip0.2cm
{\footnotesize 
\noindent
{\sc Si Duc Quang}\\
$^1$ Department of Mathematics, Hanoi National University of Education,\\
136-Xuan Thuy, Cau Giay, Hanoi, Vietnam.\\
$^2$ Thang Long Institute of Mathematics and Applied Sciences,\\
Nghiem Xuan Yem, Hoang Mai, HaNoi, Vietnam.\\
\textit{E-mail}: quangsd@hnue.edu.vn

\end{document}